\documentclass{article}

\title{Anticommutativity and the triangular lemma.}

\usepackage{amsmath,latexsym,amssymb,mathrsfs}

\date{April 2020}

\usepackage{amsmath, amssymb, amsthm, enumerate}
\usepackage[all]{xy}
\usepackage{textcomp}

\newcommand{\C}{\mathbb{C}}
\newcommand{\Set}{\mathrm{Set}}
\newcommand{\V}{\mathcal{V}}
\newcommand{\M}{\mathcal{M}}
\newcommand{\D}{\mathbb{D}}
\newcommand{\Eq}{\mathrm{Eq}}
\newcommand{\Pt}{\mathrm{Pt}}

\newcommand{\mat}[4]{\begin{pmatrix}
#1 & #3 \\
#2 & #4
\end{pmatrix}}

\newcommand{\Imp}{\mathbf{Imp}}
\newcommand{\Grp}{\mathbf{Grp}}
\newcommand{\CGrp}{\mathbf{CGrp}}
\newcommand{\Top}{\mathbf{Top}}
\newcommand{\op}{\mathrm{op}}

\theoremstyle{plain}
\newtheorem{theorem}{Theorem}[section]

\newtheorem{proposition}[theorem]{Proposition}
\newtheorem{corollary}[theorem]{Corollary}

\theoremstyle{definition}
\newtheorem{definition}[theorem]{Definition}

\newtheorem{remark}[theorem]{Remark}
\newtheorem{example}{Example}[section]

\author{Michael Hoefnagel}
\date{}
\begin{document}

\maketitle

\begin{abstract}

    For a variety $\V$, it has been recently shown that binary products commute with arbitrary coequalizers locally, i.e., in every fibre of the fibration of points $\pi: \Pt (\C) \rightarrow \C$, if and only if Gumm's shifting lemma holds on pullbacks in $\V$. In this paper, we establish a similar result connecting the so-called triangular lemma in universal algebra with a certain categorical \emph{anticommutativity} condition. In particular, we show that this anticommutativity and its local version are Mal'tsev conditions, the local version being equivalent to the triangular lemma on pullbacks.  As a corollary, every locally anticommutative variety $\V$ has directly decomposable congruence classes in the sense of Duda, and the converse holds if $\V$ is idempotent. 

\end{abstract}

\section{Introduction} \label{sec:introduction}

Recall that a category is said to be \emph{pointed} if it admits a \emph{zero object} $0$, i.e., an object which is both initial and terminal. For a variety $\V$, being pointed is equivalent to the requirement that the theory of $\V$ admit a unique constant. Between any two objects $X$ and $Y$ in a pointed category, there exists a unique morphism $0_{X,Y}$ from $X$ to $Y$ which factors through the zero object. The presence of these \emph{zero morphisms} allows for a natural notion of kernel or cokernel of a morphism $f\colon X \rightarrow Y$, namely, as an equalizer or coequalizer of $f$ and $0_{X,Y}$, respectively. Every kernel/cokernel is a monomorphism/epimorphism, and a monomorphism/epimorphism is called \emph{normal} if it is a kernel/cokernel of some morphism. Given any pointed category $\C$ with binary products, the product inclusion $X \xrightarrow{\iota_1} X \times Y$ (i.e., the unique morphism into $X\times Y$ determined by $1_X$ and $0_{X,Y}$) is always a kernel of the product projection $X \times Y \xrightarrow{\pi_2} Y$, but it is not generally true that $\pi_2$ is a cokernel of $\iota_1$. Pointed categories where every product projection is normal are said to have  \emph{normal projections} \cite{Janelidze2003}. For example, every subtractive \cite{Janelidze2005} or unital category \cite{Bou96} has normal projections. In particular, every pointed subtractive variety or any J\'onsson-Tarski variety has normal projections. 

To any object $X$, in any category $\C$, we may associate the pointed category $\Pt_{\C}(X)$ of split epimorphisms with codomain $X$ and chosen splitting --- the so-called ``category of points'' of $X$. Explicitly, objects of this category are triples $(A,p,s)$ where $A$ is an object of $\C$ and $p:A \rightarrow X$ and $s:X \rightarrow A$ are morphisms in $\C$ with $p \circ s = 1_X$. A morphism $f:(A,p,s) \rightarrow (B,q,t)$ in $\Pt_{\C}(X)$ is a morphism $f:A \rightarrow B$ in $\C$ such that $q \circ f = p$ and $f \circ s = t$, i.e., it is a morphism which makes the upward and downward facing triangles in the diagram below commute.
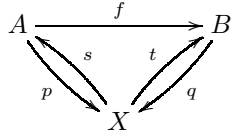
\[
\xymatrix{
A \ar[rr]^-f \ar@/_0.3pc/[dr]_-p &  & B \ar@/^0.3pc/[dl]^-q \\
& X \ar@/^0.3pc/[ur]^-t \ar@/_0.3pc/[ul]_s
}
\]
The categories $\Pt_{\C}(X)$ arise as the fibres of a certain functor $\pi: \Pt(\C) \rightarrow \C$, namely,  the \emph{fibration of points} \cite{Bou96}. This functor plays an important role in categorical algebra, as it classifies many notions of central importance to the subject such as, for example, the notion of a Mal'tsev category \cite{CarLamPed91, CarPedPir92}: a finitely complete category $\C$ is Mal'tsev ---  i.e., $\C$ satisfies the categorical condition associated with Mal'tsev varieties, if and only if $\Pt_{\C}(X)$ is unital \cite{Bou96} --- i.e., $\Pt_{\C}(X)$ satisfies the categorical condition associated with J\'onsson-Tarski varieties. 
There is a canonical zero object in $\Pt_{\C}(X)$ which is given by $(X,1_X,1_X)$, and with respect to this zero object the zero morphism from $(A,p,s)$ to $(B,q,t)$ is simply the composite $t \circ p$. If $\C$ has pullbacks, then we may form products in $\Pt_{\C}(X)$ in the following way. Consider a pullback $A \times_X B$ of $p$ along $q$:
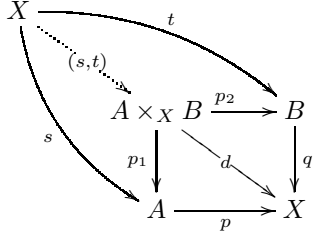
\[
\xymatrix{
& X \ar@{..>}[rd]|{(s,t)}\ar@/^1pc/@{->}[rrd]^{t} \ar@/_1pc/@{->}[rdd]_{s} & & \\
& & A \times_X B \ar[r]^{p_2} \ar[d]_{p_1} \ar[dr]|d & B \ar[d]^{q} \\
& & A \ar[r]_{p} & X
}
\]
Since $p \circ s = 1_X = q \circ t$, the universal property of the pullback determines a unique morphism $(s,t)$ with the property that $p_1 \circ (s,t) = s$ and $p_2 \circ (s,t) = t$. Then it is readily checked that the diagram 
\[
(A,p,s) \xleftarrow{p_1} (A \times_X B, d, (s,t)) \xrightarrow{p_2} (B,q,t),
\]
is a product diagram in $\Pt_{\C}(X)$. Now, a finitely complete category $\C$ is then said to have \emph{normal local projections} \cite{Janelidze2004} if for every object $X$ in $\C$, the category $\Pt_{\C}(X)$ has normal projections. For varieties of algebras, it has been recently shown in \cite{Hoefnagel2019c} that normal local projections is closely related to the (local) commutativity of binary products with arbitrary coequalizers, which in turn is closely related to Gumm's \emph{shifting lemma} \cite{Gum83}. Recall that an algebra $X$ satisfies the shifting lemma if for any congruences $R,S,T$ on $X$ such that $R \cap S \leqslant T$ and $x R u, y R v$ and $x S y, u S v$, then $u T v$ implies $x T y$. A variety is said to satisfy the shifting lemma if every algebra in it satisfies the shifting lemma, which happens if and only if the variety is congruence modular. The implications of relations in the shifting lemma are usually depicted with a diagram:
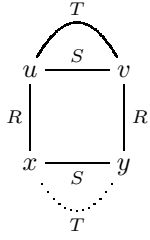
\[
\xymatrix{
	u \ar@/^1.5pc/@{-}[r]^{T} \ar@{-}[r]^{S} \ar@{-}[d]_{R} & v \ar@{-}[d]^{R} \\
	x \ar@/_1.5pc/@{..}[r]_{T} \ar@{-}[r]_{S} & y
}
\]
Given morphisms $p:A \rightarrow X$ and $q:B \rightarrow X$ in a variety of algebras $\V$ then their pullback is given by $A \times_X B = \{(a,b) \mid p(a) = q(b)\}$ together with the canonical projection morphisms $p_1:(a,b) \mapsto a$ and $p_2:(a,b) \mapsto b$. The result alluded to earlier connecting normal local projections with the shifting lemma, is that a variety $\V$ has normal local projections if and only if $\V$ \emph{satisfies the shifting lemma on pullbacks}: for any morphisms $p:A \rightarrow X$ and $q:B \rightarrow X$ in $\V$ and any congruence $\Theta$ on $A \times_X B$ we have $(x,u) \Theta (y, u) \Rightarrow (x,v) \Theta (y, v)$ for any elements $(x,u), (y, u), (x,v), (y, v)$ of $A \times_X B$. This can be seen as the shifting lemma in the special case where $R = \Eq(p_1)$ and $S = \Eq(p_2)$ are the kernel congruence relations of the morphisms $p_1$ and $p_2$, respectively.
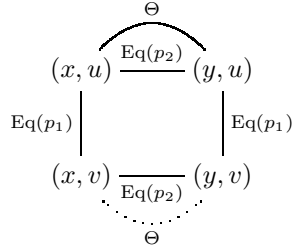
\[
\xymatrix{
	(x,u) \ar@/^1.5pc/@{-}[r]^{\Theta} \ar@{-}[r]^{\Eq(p_2)} \ar@{-}[d]_{\Eq(p_1)} & (y,u) \ar@{-}[d]^{\Eq(p_1)} \\
	(x,v) \ar@/_1.5pc/@{..}[r]_{\Theta} \ar@{-}[r]_{\Eq(p_2)} & (y,v) 
}
\]
This is because we always have $ \Eq(p_1) \cap \Eq(p_2) \subseteq \Theta$. Moreover, a variety $\V$ satisfies the shifting lemma on pullbacks if and only if finite products commute with arbitrary coequalizers in $\Pt_{\V}(X)$ for any algebra $X$ in $\V$ (see Theorem~3.7 in \cite{Hoefnagel2019c}). Recall that an algebra $X$ in a variety is said to satisfy the \emph{triangular lemma} \cite{Duda2000} (see also \cite{Chajda2003}) if for any three congruences $R, S, T$ on $X$ such that $R \cap S \leqslant T$, if $x R y S z$ and $x T z$, then $y T z$. 
	\[
	\xymatrix{
		& z \ar@{-}[d]^{S} \\
		x  \ar@{-}[ur]^{T} \ar@{-}[r]_{R}  & y \ar@{..}@/_1pc/[u]_{T} 
	}
	\]
The main aim of this paper is to illustrate a similar result connecting the triangular lemma with a natural \emph{anticommutativity} condition (Definition~\ref{def:anticommutative}) which is based the notion of commuting morphisms is the sense of Huq \cite{Huq1968} (see Definition~\ref{def:anticommutative} below). In particular, we show that for varieties, this notion of anticommutativity, as well as its local version, are both Mal'tsev properties, the latter being equivalent to the triangular lemma restricted to pullbacks: a variety $\V$ is \emph{locally anticommutative} if and only if for any two morphisms $p\colon A \rightarrow X$ and $q\colon B \rightarrow X$, and any congruence $\Theta$ on the pullback $A \times_X B$ of $p$ along $q$, if $(x,y), (x',y), (x',y') \in A \times_X B$ are any elements such that $(x,y) \Theta (x',y')$, then $(x',y) \Theta (x',y')$. 
\[
\xymatrix{
	& (x',y') \ar@{-}[d]|{\Eq(p_1)} \ar@{..}@/^2pc/[d]^{\Theta} \\
	(x,y)    \ar@{-}[ur]^{\Theta} \ar@{-}[r]_{\Eq(p_2)}  & (x',y) 
}
\]

\subsubsection*{Internal relations in categories} Given any two objects $X$ and $Y$ in a category $\C$ with binary products, recall that an internal binary relation $R$ from $X$ to $Y$ is simply a monomorphism $r=(r_1,r_2):R_0 \rightarrow X \times Y$, and the morphisms $r_1$ and $r_2$ are called the \emph{projections} of the relation. In what follows, we will say that $r$ represents the relation $R$. Throughout the rest of this paper, we will adhere to the following notation: given morphisms $x:S \rightarrow X$ and $y:S \rightarrow Y$ in $\C$ we write $x R y$ or $(x,y) \in_S R$ if the morphism $(x,y):S \rightarrow X \times Y$ factors through $r$, i.e., there exists a (necessarily unique) morphism $u:S \rightarrow R_0$ such that $r \circ u = (x,y)$. This notation extends to relations of higher arity: if $R$ is an internal $n$-ary relation $r:R_0 \rightarrow X_1 \times \cdots \times X_n$, then for any morphisms $x_i:S \rightarrow X_i$ for $i=1,\dots,n$, we write $(x_1,x_2,\dots,x_n) \in_S R$ if the morphism $( x_1,x_2,\dots,x_n):S \rightarrow X_1 \times \cdots \times X_n$ factors through $r$. For two internal relations $R,R'$ between objects $X_1,...,X_n$, we write $R \leqslant R'$ if the representing monomorphism of $R$ factors through the representing monomorphism of $R'$. With respect to this notation, we can reformulate familiar properties of relations internal to categories. For example an internal binary relation $R$ on an object $X$ in $\C$ is transitive if for any $x,y,z:S \rightarrow X$ we have $x R y$ and $y R z$ implies $x R z$. In a similar way we can reformulate the property of a binary relation to be reflexive, symmetric or an equivalence, internal to categories. Moreover, we can also reformulate various shifting properties, such as the triangular lemma, in general categories (see Definition~\ref{def:triangle-lemma} below). If $\C$ has pullbacks, then the intersection of two relations can be defined. For example, if $R$ and $T$ are binary relations represented by the monomorphisms $r=(r_1,r_2):R_0 \rightarrow X \times Y$ and $t=(t_1,t_2):T_0 \rightarrow X \times Y$ respectively, then their intersection $R \cap T$ is represented by the diagonal morphism in any pullback
\[
\xymatrix{
(R\cap T)_0 \ar[r] \ar[d] \ar[rd] & T_0  \ar[d]^{t}\\
R_0 \ar[r]_r& X \times Y
}
\]
since the diagonal morphism in a pullback as above is always a monomorphism. This relation is then easily seen to posses the property $(x,y) \in_S R \cap T$ if and only if $(x,y) \in_S R$ and $(x,y)\in_S T$ for any $x:S \rightarrow X$ and $y:S \rightarrow Y$. 
\begin{remark}
This technique of working ``set-theoretically'' within abstract categories is the standard one of working with ``generalized elements'', which is formally validated by the Yoneda embedding $Y: \C \rightarrow \Set ^{\C^{\op}}$. 
\end{remark}
\begin{definition}\label{def:triangle-lemma}
	Let $\C$ be a category with binary products. An object $X$ in a category $\C$ satisfies the \emph{triangular lemma} if for any three (internal) equivalence relations $R, S, T$ on the same object $X$ such that $R \cap S\leqslant T$, if $xRySz$ and $x T z$, then $yTz$, for any morphisms $x,y,z:S \rightarrow X$. Then $\C$ is said to satisfy the triangular lemma if every object in $\C$ does. 
\end{definition}
\noindent
Given a variety $\V$, by making use of the free algebra over one generator in  $\V$ it is an easy exercise to show that $\V$ satisfies the triangular lemma in the sense of Definition~\ref{def:triangle-lemma} if and only if it satisfies the triangular lemma in the universal algebraic sense. Any congruence distributive variety satisfies the triangular lemma, but not every variety satisfying the triangular lemma is congruence distributive as shown in \cite{Chjada2007}. However, a congruence modular variety is congruence distributive if and only if it satisfies the triangular lemma \cite{Chjadaetal2003}. But also, any regular \emph{majority category} in the sense of \cite{Hoe18a} satisfies the triangular lemma in the sense of Definition~\ref{def:triangle-lemma} (see Lemma~1.1 and the paragraph preceding it in \cite{GranRodeloNgeufue2019}). Thus for example, the dual category $\Top^{\op}$ is a regular majority category (Example~2.6 in \cite{Hoe18a}), and hence it satisfies the triangular lemma. 

\begin{description}
\item[Convention]  We will sometimes denote zero morphisms between objects in a pointed category simply by $0$, leaving out the domain and codomain reference as subscripts.
\end{description}
\section{Anticommutative categories} \label{sec:anticommutativity}

Recall that two morphisms $f\colon A \rightarrow C$ and $g\colon B \rightarrow C$ in a pointed category $\C$ with binary products are said to \emph{commute} \cite{Huq1968}  (or ``cooperate''  \cite{Bou02}) if there exists a morphism $\rho\colon A \times B \rightarrow C$ such that $\rho \circ \iota_1 = f$ and $\rho \circ \iota_2 = g$, where $\iota_1\colon A \rightarrow A \times B$ and $\iota_2\colon  B \rightarrow A \times B$ are the canonical product inclusions. Following the terminology of \cite{Bou02}, we will call such a morphism $\rho$ a \emph{cooperator} for $f$ and $g$ in what follows.
\[
\xymatrix{
A \ar[rd]_{f} \ar[r]^-{\iota_1} & A \times B \ar@{..>}[d]_-\rho & B \ar[l]_-{\iota_2} \ar[ld]^{g} \\
& C & 
}
\]
A morphism $f\colon  X \rightarrow Y$ in $\C$ is called \emph{central} when it commutes with the identity $1_Y$ on $Y$, and an object $M$ is called \emph{commutative} if $1_M$ is central. 
For example, in the category $\Grp$ of groups, two morphisms $f\colon G \rightarrow L$ and $g\colon H \rightarrow L$ commute if and only the subgroups $f(G)$ and $g(H)$ commute in the usual sense. The category $\Imp$ of (non-empty) \emph{implications algebras} \cite{Nem65} is pointed, and if two morphisms $f:A \rightarrow C$ and $g:B \rightarrow C$ commute then $f(A) \cap g(B) = 0$. In other words, if two morphisms $f$ and $g$ of implication algebras commute, then they are \emph{disjoint} in the following sense: two morphisms $f\colon X \rightarrow Z$ and $g\colon Y \rightarrow Z$ in a pointed category $\C$ are said to be disjoint if for any commutative diagram
\[
\xymatrix{
	S \ar[r]^-{x} \ar[d]_{y}  & Y \ar[d]^g \\
	X \ar[r]_f & Z
}
\] 
we have $g\circ x = 0 =  f\circ y$. This brings us to the main definition of this paper: 
\begin{definition} \label{def:anticommutative}
	A pointed category $\C$ with binary products is a called \emph{anticommutative} if every pair of commuting morphisms are disjoint.
\end{definition}

\begin{proposition} \label{prop:reformulation-kernels-and-pullbacks}
	Let $\C$ be a pointed category with binary products and kernels, and suppose that $f\colon X \rightarrow Z$ and $g\colon  Y \rightarrow Z$ are any pair of morphisms in $\C$. Then $f$ and $g$ are disjoint if and only if the diagram
	\[
	\xymatrix{
		\ker(f) \times \ker(g) \ar[d]_{p_1} \ar[r]^-{p_2} & Y \ar[d]^g \\
		X \ar[r]_f & Z
	}
	\]
	is a pullback, where $p_1$ and $p_2$ are the canonical product projections composed with the canonical kernel inclusions.
\end{proposition}
\begin{proof}
	If $f$ and $g$ are disjoint, and $x\colon S \rightarrow X$ and $y\colon S \rightarrow Y$ are any morphisms such that $f \circ x = g \circ y$, then anticommutativity implies that $f\circ x = 0 = g\circ y$ so that $x$ and $y$ factor through $\ker(f)$ and $\ker(g)$ respectively, and hence $(x,y)$ factors through $\ker(f) \times \ker(g)$. Conversely, if the diagram above is a pullback, then $f\circ x = g\circ y$ gives a factorization of both $x$ and $y$ through $\ker(f)$ and $\ker(g)$ respectively, so that $f\circ x = 0 = g\circ y$. 
\end{proof}
\begin{remark} \label{rem:pointed-majority-category-anticommutative}
	The content of the above proposition was essentially what was used in \cite{Hoe18a} in order to show that every pointed finitely complete majority category is anticommutative (see Proposition~3.8). 
\end{remark}
\noindent
Recall that a category $\C$ is said to be \emph{regular} \cite{BGO71} it satisfies the following three properties: 
\begin{enumerate}[(i)]
    \item $\C$ has all finite limits.
    \item $\C$ has coequalizers of kernel-pairs. 
    \item The class of all regular epimorphisms in $\C$ is pullback stable, i.e., for any pullback diagram 
    \[
    \xymatrix{
    \bullet \ar[r] \ar[d]_p & \bullet \ar[d]^e\\
    \bullet \ar[r] & \bullet 
    }
    \]
    if $e$ is a regular epimorphism, then so is $p$. 
\end{enumerate}
For example, any (quasi)variety of algebras is a regular category. Every morphism $f:X \rightarrow Y$ in a regular category $\C$ admits a factorization $f = m \circ e$ where $e:X \rightarrow I$ is a regular epimorphism and $m$ is a monomorphism.  Such a factorization of $f$ is called an \emph{image-factorization} of $f$ and is unique up to unique isomorphism. If $\C$ is a variety of algebras, then the usual projection $X \rightarrow f(X)$ followed by the inclusion $f(X) \rightarrow Y$ yields an image factorization for $f$. In what follows, we will write $X \xrightarrow{e_f} f(X) \xrightarrow{m_f}Y$ for a chosen image factorization of a morphism $f:X \rightarrow Y$ in a regular category. 
\begin{proposition} \label{prop:regular-characterization}
	For any pointed regular category $\C$ and any morphisms $f\colon X \rightarrow Z$ and $g\colon Y \rightarrow Z$ in $\C$, $f$ and $g$ are disjoint if and only if $m_f$ and $m_g$ are disjoint.
\end{proposition}
\noindent
Note that according to Proposition~\ref{prop:reformulation-kernels-and-pullbacks}, two monomorphisms in a pointed category with binary products and kernels are disjoint if and only if their pullback is a zero-object.
\begin{proof}
	Consider the diagram below where the bottom row and right column are image factorizations of $f$ and $g$ respectively and each square is a pullback.
	\[
	\xymatrix{
		P \ar@{->>}[rd]^e \ar@{->>}[r]^{p_2} \ar@{->>}[d]_{p_1} & K_2 \ar@{->>}[d] \ar[r]^{k_2} & Y \ar@{->>}[d]^{e_g}\\
		K_1 \ar[d]_{k_1} \ar@{->>}[r] & I \ar[rd]^{ m }\ar[d] \ar[r] & g(Y) \ar[d]^{m_g} \\
		X \ar@{->>}[r]_{e_f} & f(X) \ar[r]_{m_f} & Z
	}
	\]
	In the above diagram the diagonal morphism $e$ is a regular epimorphism (as it is a composite of regular epimorphisms in a regular category) and the morphism $m$ is a monomorphism. Therefore, if $f$ and $g$ are disjoint then $m \circ e = 0$ which implies that $m = 0$ so that $m_f$ and $m_g$ are disjoint. On the other hand, if $m_f$ and $m_g$ are disjoint then $I$ is a zero object, and hence each square being a pullback implies that $K_1 = \ker(f)$ and $K_2 = \ker(g)$ and that $P = \ker(f) \times \ker(g)$, $p_1$ and $p_2$ are the canonical product projections, $k_1$ and $k_2$ are the kernel inclusions, so that $f$ and $g$ are disjoint by Proposition~\ref{prop:reformulation-kernels-and-pullbacks}. 
\end{proof}
\noindent
\begin{proposition} \label{prop:internal-monoids-trivial}
	For a pointed category $\C$ with binary products, we have $(i) \implies (ii) \implies (iii)$ where:
	\begin{enumerate}[(i)]
		\item $\C$ is anticommutative.
		\item Every central morphism in $\C$ is a zero morphism. 
		\item Every commutative object in $\C$ is a zero object.
	\end{enumerate}
\end{proposition}
\begin{proof}
	For $(i) \implies (ii)$, note that if $f\colon X \rightarrow Y$ is central, then it commutes with the identity $1_Y$, so that $f$ and $1_Y$ are disjoint, so that $f = 0$. Note that $(ii) \implies (iii)$ is trivial. 
\end{proof}
\noindent
A pointed finitely complete category $\C$ is called \emph{unital} \cite{Bou96} if for any commutative diagram 
\[
\xymatrix{
& \bullet \ar[d]^m & \\
A \ar[r]_-{\iota_1} \ar[ru] & A \times B & B \ar[l]^-{\iota_2} \ar[lu]
}
\]
where $\iota_1$ and $\iota_2$ are the canonical product inclusions, if $m$ is a monomorphism then $m$ is an isomorphism. This is equivalent to requirement that the product inclusions $\iota_1$ and $\iota_2$ are \emph{jointly strongly epimorphic}.  A variety is then unital in the above sense if and only if it is a J\'onsson-Tarski variety. The corollary below shows that the converse implications in the statement of Proposition~\ref{prop:internal-monoids-trivial} hold when the base category is regular and unital. 
\begin{corollary} \label{cor:triv-moinoid-characterization}
	For a regular unital category $\C$ the following are equivalent: 
	\begin{enumerate}[(i)]
		\item $\C$ is anticommutative.
		\item Every central morphism in $\C$ is $0$. 
		\item Every commutative object in $\C$ is trivial.
	\end{enumerate}
\end{corollary}

\begin{proof}
The implications $(i) \implies (ii) \implies (iii)$ follow from Corollary~\ref{prop:internal-monoids-trivial}.
	For $(iii) \implies (i)$, we note that if $f\colon X \rightarrow Z$ and $g\colon Y \rightarrow Z$ are two morphisms which commute, then $\C$ being unital, their images $f(X)$ and $g(Y)$ commute. Their intersection $f(X) \cap g(Y)$ is a commutative object in $\C$ (see Example~1.4.2 in \cite{BorceuxBourn}), and hence, $f(X) \cap g(Y) = 0$ so that $f$ and $g$ are disjoint by Proposition~\ref{prop:regular-characterization}. 
\end{proof}
If $\C$ is not regular, then $(ii)$ need not imply $(i)$. For example, consider the full subcategory $\CGrp$ of $\Grp$ consisting \emph{centerless groups}, i.e., groups $G$ for which $\mathrm{Z}(G) = \{x \in G \mid \forall y \in G (xy = yx)\} = 0$. The category $\CGrp$ has products, is unital (product inclusions are jointly strongly epimorphic) and every central morphism in $\CGrp$ is a zero morphism, however it is not anticommutative. To see this, consider the free group $F$ over two generators $x,y$, and consider the morphism $f\colon F \rightarrow F$ such that $f(x) = x = f(y)$. Note that $F$ is centerless, so that $f$ is a morphism in $\CGrp$. Moreover, $f$ commutes with itself, since the map $\rho\colon F^2 \rightarrow F$ defined by $(a,b) \mapsto f(a)\cdot f(b)$ is a cooperator for $f$ with itself, but $f$ is not disjoint with itself.

\begin{proposition} \label{prop:characterization-one}
	For a pointed category $\C$ with binary products the following are equivalent. 
	\begin{enumerate}[(i)]
		\item For any morphism $\rho\colon  X \times X \rightarrow Y$ if $\rho \circ \iota_1 = \rho \circ \iota_2$ then $\rho \circ \iota_1 = 0 = \rho \circ \iota_2$.
		\item For any morphism $\rho\colon X \times X \rightarrow Y$ we have $\rho \circ (x,0) = \rho \circ (0,x) \implies \rho \circ (x,0) = 0 = \rho \circ (0, x)$ for any morphism $x\colon S \rightarrow X$. 
		\item For any morphism $\rho\colon  X \times Y \rightarrow Z$ we have $\rho\circ (x,0) = \rho \circ (0,y) \implies \rho \circ (x,0) = 0  = \rho\circ (0,y)$ for any morphisms $x\colon S \rightarrow X$ and $y\colon S \rightarrow Y$.
		\item $\C$ is anticommutative. 
	\end{enumerate}
\end{proposition}
\begin{proof}
	For $(i) \implies (ii)$, let $x\colon S \rightarrow X$ be any morphism such that $\rho\circ (x,0) = \rho\circ (0,x)$. Consider the composite morphism $S \times S \xrightarrow{x\times x} X \times X \xrightarrow{\rho} Y$. Then we have that $(\rho \circ (x \times x)) \circ \iota_1 = \rho \circ (x,0) = \rho \circ (0,x) = (\rho \circ (x\times x)) \circ \iota_2$, which implies that:
	\[
	\rho \circ (x,0) = (\rho \circ (x \times x)) \circ \iota_1 = 0 = (\rho \circ (x\times x)) \circ \iota_2 = \rho \circ (0,x).
	\]
	For $(ii) \implies (iii)$ suppose that $x\colon S \rightarrow X$ and $y\colon  S \rightarrow Y$ are any morphisms such that $\rho \circ (x,0) = \rho\circ (0,y)$. Consider the morphism $(X\times Y) \times (X \times Y) \xrightarrow{p} Z \times Z$ which is defined by the map $((a,b), (c,d)) \longmapsto (\rho (a,d), \rho (c,b))$. Then we have that
	\begin{align*}
	p\circ ((x,y), (0,0)) &= (\rho \circ (x,0), \rho\circ (0,y)) \\
	&= (\rho\circ (0,y), \rho\circ (x,0)) \\
	&= p\circ ((0,0), (x,y)),	    
	\end{align*}
	which implies $p\circ ((x,y), (0,0)) = (\rho \circ (x,0), \rho \circ (0,y)) = (0,0)$ and the result follows. For $(iii) \implies (iv)$ suppose that $\rho$ is a cooperator for two morphisms $f\colon X \rightarrow Z$ and $g\colon  Y \rightarrow Z$, and let $x:S \rightarrow X$ and $y:S \rightarrow Y$ be any two morphisms such that $f \circ x = g \circ y$. Consider the diagram:
	\[
	\xymatrix{
		& X \ar[d]^{\iota_1} \ar[dr]^{f}&  & \\
		S \ar[ur]^x \ar[dr]_y& X \times Y \ar[r]^{\rho} & Z \\
		& Y\ar[u]_{\iota_2} \ar[ru]_{g}& 
	}
	\]
	The commutativity of outer square gives 
	\[
	\rho \circ (x,0) = f\circ x = g \circ y = \rho \circ (0,y) \implies f\circ x = 0 = g \circ y. 
	\]
For $(iv) \implies (i)$, just note that $\rho$ is a cooperator for $\rho \circ \iota_1$ and $\rho \circ \iota_2$
\end{proof}
\noindent
The following proposition is a simple reformulation of $(i)$ in the above proposition when $\C$ has coequalizers.
\begin{proposition} \label{prop:characterization-with-coequalizers}
	A pointed category $\C$ with binary products and coequalizers is anticommutative if and only if for any object $X$ in $\C$ we have $q \circ \iota_1 = 0 = q \circ \iota_2$ for any coequalizer $q: X \times X \rightarrow Q$ of $\iota_1$ and $\iota_2$.  
\end{proposition}
\subsection{Examples of anticommutative categories}
Given any morphism $f:X \rightarrow Y$ in a finitely complete category, the kernel equivalence relation $\Eq(f)$ of $f$ is the equivalence relation obtained by pulling back $f$ along itself. If $\C$ is a variety of algebras, then $\Eq(f)$ is the kernel congruence associated to $f$, i.e., we have: 
\[
\Eq(f) = \{(x,y) \in X \times X \mid f(x) = f(y)\}.
\]
\begin{example} \label{example:anti-commutative}
	If $\V$ is a pointed variety of universal algebras which admits an idempotent binary operation $b(x,y)$ satisfying $b(x,0) = 0 = b(0,y)$, then $\V$ satisfies (ii) of Proposition~\ref{prop:characterization-one}. Since for any homomorphism $f\colon X\times X \rightarrow Y$ in $\V$ if $f(x,0) = f(0,x)$ then $f(x,0) = b(f(x,0), f(x,0))= b(f(x,0), f(0,x)) = f(b(x,0),b(0,x)) = 0$. Thus, the variety of meet-semilattices with a least element is anticommuative. In Theorem~\ref{thm:Mal'tsev-characterization} below, a Mal'tsev condition is given for a pointed variety to be anticommutative.
\end{example}
\noindent
A variety $\V$ is said to have \emph{directly decomposable congruence classes} \cite{Duda1986} if every congruence class $C$ on a product $X \times Y$ in $\V$ is such that $C = \pi_1(C) \times \pi_2(C)$. This property is easily seen to be equivalent to the requirement that every congruence $\Theta$ on any product $X \times Y$ in $\V$, satisfies the implication: $$(x,y)\Theta (x', y') \implies (x',y) \Theta (x',y').$$ Moreover, this may be viewed as the triangular lemma restricted to products:
	\[
	\xymatrix{
		& (x',y') \ar@{-}[d]|{\Eq(\pi_1)} \\
		(x,y)  \ar@{-}[ur]^{\Theta} \ar@{-}[r]_{\Eq(\pi_2)}  & (x',y) \ar@{..}@/_2pc/[u]_{\Theta} 
	}
	\]
\begin{proposition} \label{prop:directly-decomposable-congruence-class-anti-commutative}
Any pointed variety which has directly decomposable congruence classes is anticommutative.
\end{proposition}
\begin{proof}
We verify the conditions of  Proposition~\ref{prop:characterization-one} $(ii)$. Suppose that $f\colon X \times X \rightarrow Y$ is any morphism in $\V$ and that $f(x,0) = f(0,x)$,  then $(x,0)\Eq(f)(0,x)$ implies $(0,0)\Eq(f) (0,x)$.
	\[
	\xymatrix{
		& (0,x) \ar@{-}[d]|{\Eq(\pi_1)} \\
		(x,0)  \ar@{-}[ur]^{\Eq(f)} \ar@{-}[r]_{\Eq(\pi_2)}  & (0,0) \ar@{..}@/_2pc/[u]_{\Eq(f)} 
	}
	\]
\end{proof}
\begin{remark} \label{rem:pointed-traingular-lemma-anticommutative}
Essentially the same argument, working categorically may be used to show that any pointed finitely complete category which satisfies the triangular lemma is anticommutative.
\end{remark}
\noindent
The notion of a \emph{majority category} has been recently introduced and studied in \cite{Hoe18a, Hoe18b}. This notion is thought to be a categorical counterpart for varieties which admit a majority term, in a similar way the notion of Mal'tsev category \cite{CarPedPir92} is a categorical counterpart of Mal'tsev varieties. A category $\C$ is a majority category if for any ternary relation $R$ between objects $X,Y,Z$ we have: $(x',y,z) \in_S R$ and $(x,y',z) \in_S R$ and $(x,y,z') \in_S R$ implies $(x,y,z)\in_S R$, for any morphisms $x,x'\colon S \rightarrow X$ and $y,y'\colon S \rightarrow Y$ and $z,z'\colon S \rightarrow Z$ in $\C$. According to Proposition~\ref{prop:reformulation-kernels-and-pullbacks} above, and Proposition~3.8 in \cite{Hoe18a}, it follows that every pointed finitely complete majority category is anticommutative. However, for completeness we include a short proof here: 
\begin{proposition} \label{prop:pointed-majority-anticommutative}
Every pointed finitely complete majority category is anticommutative.
\end{proposition}
\begin{proof}
We will show that $(i)$ of Proposition~\ref{prop:characterization-one} holds. Let $f:X\times X \rightarrow Y$ be any morphism with $f\circ \iota_1 = y = f \circ \iota_2$. Let $R$ be the ternary relation defined by the monomorphism $r:R_0 \rightarrow X \times Y \times X$ where $r$ is the equalizer of  $f\circ (\pi_1, \pi_3)$ and $\pi_2$. 
\[
\xymatrix{
R_0 \ar[r]^-r &   X \times Y \times X \ar@<-.8ex>[rr]_-{\pi_2} \ar@<.8ex>[rr]^-{f\circ (\pi_1, \pi_3)} & &Y
}
\]
Then we have 
$(\iota_1,y,0) \in_X R$ and $(0,0,0) \in_X R$ and $(0,y,\iota_2) \in_X R$ which implies  $(0,y,0)\in_X R$ by the majority property. Therefore, $f \circ \iota_1 = y = 0$.
\end{proof}
\noindent
The notion of an \emph{$\M$-coextensive} object \cite{Hoe2020a} in a category $\C$  is an object-wise coextensivity \cite{CLW93} condition relative to a class of morphisms $\M$ in $\C$. Let $\C$ be a pointed category with binary products and coequalizers. When $\M$ is the class of regular epimorphisms in $\C$, then an object $X$ is $\M$-coextensive if and only if for any diagram
\[
\xymatrix{
X_1 \ar[d]_{q_1} & X \ar[l] \ar[d]^{q} \ar[r] \ar[d] & X_2 \ar[d]^{q_1} \\
Q_1 & Q \ar[r] \ar[l] & Q_2 
}
\]
where the top row is a product diagram and the vertical morphisms are regular epimorphisms, then the bottom row is a product diagram if and only if both the squares above are pushouts. For example, given any algebra $X$ in any variety, then $X$ is $\M$-coextensive (where $\M$ is the class of regular epimorphisms in the variety) if and only it has the \emph{Fraser-Horn property} \cite{FraserHorn1970}. We say that $\C$ is regularly-coextensive if every object in $\C$ is $\M$-coextensive with $\M$ the class of all regular epimorphisms in $\C$. 
\begin{proposition} \label{prop:reg-coextensive-anticommutative}
If $\C$ is a pointed category with binary products and coequalizers which is regularly coextensive, then $\C$ is anticommutative.
\end{proposition}
\begin{proof}
We will show that $\C$ satisfies the conditions of Proposition~\ref{prop:characterization-with-coequalizers}. Let $X$ be any object in $\C$, and suppose that $q:X \times X \rightarrow Q$ is a coequalizer of the product inclusions $\iota_1:X \rightarrow X \times X$ and $\iota_2:X \rightarrow X \times X$. Then the pushout of $q$ along $\pi_1$ and along $\pi_2$ is formed simply by taking a coequalizer of $\pi_1 \circ \iota_1, \pi_1 \circ \iota_2$ and $\pi_2 \circ \iota_1, \pi_2 \circ \iota_2$ respectively. But these two pairs of coequalizers have terminal objects for codomains, and hence $Q$ being their product, is terminal. 
\end{proof}
\subsection{Locally anticommutative categories} \label{section:locally-anticommutative-categories}
For an object $X$ in a category, recall from the introduction that the fibre $\Pt_{\C}(X)$ above $X$ of the fibration of points $ \pi\colon \Pt(\C) \rightarrow \C$ consists of triples $(A,p,s)$ where $p\colon A \rightarrow X$ is a split epimorphism in $\C$ and $s$ is a splitting for $p$. A morphism $f\colon (A,p,s) \rightarrow (B,q,t)$ in $\Pt_{\C}(X)$ is a morphism $f\colon A \rightarrow B$ in $\C$ such that $q \circ f = p$ and $f \circ s = t$. The category $\Pt_{\C}(X)$ is always pointed, where the zero-object is $(X,1_X, 1_X)$, and if $\C$ is finitely complete, then so is $\Pt_{\C}(X)$. 
\begin{definition} \label{def:locally-anticommutativity}
A category $\C$ is \emph{locally anticommutative} if for any object $X$ in $\C$, the category $\Pt_{X}(\C)$ is anticommutative.
\end{definition}
\begin{proposition} \label{prop:points-triangular-lemma-persevation-reflection}
	If $\D$ is any finitely complete category which satisfies the triangular lemma, and $F\colon  \C \rightarrow \D$ is any conservative functor (,i.e., reflects isomorphisms,) which preserves pullbacks and equalizers then $\C$ satisfies the triangular lemma. 
\end{proposition}
\noindent
Note that the assumptions on the functor $F$ imply that it preserves monomorphisms, and that if $E$ is an equivalence relation in $\C$, then $F(E)$ -- the relation obtained by applying $F$ to the representative of $E$ -- is an equivalence relation in $\D$. We provide a sketch of the proof below, which is a standard preservation/reflection argument. 
\begin{proof}[Sketch]
	Suppose that $R,S,T$ are equivalence relations on an object $X$ in $\C$ such that $R \cap S \leqslant T$ and let $x,y,z\colon S \rightarrow X$ be morphisms as in Definition~\ref{def:triangle-lemma}, where $xRySz$ and $xTz$.  Then we are required to show that $yTz$, which is equivalent to showing that in the pullback diagram
	\[
	\xymatrix{
		P \ar[r]^{p_1} \ar[d]_{p_2} & T \ar[d]^t \\
		S \ar[r]_{(y,z)} & X \times X
	}
	\]
	$p_2$ is an isomorphism. Applying $F$ to the diagram above, we obtain a pullback diagram in $\D$. The assumptions on $F$ easily imply that the canonical morphism $F(X \times X) \rightarrow F(X) \times F(X)$ is a monomorphism, which implies that $(F(P), F(p_1), F(p_2))$ form a pullback of $F(t)$ along $(F(y), F(z))$. Since $F(x)F(R)F(y)F(S)F(z)$ and $F(x)F(T)F(z)$ and since $\D$ satisfies the triangular lemma, $(F(y), F(z))$ factors through $T$, which implies that $F(p_2)$ is an isomorphism, so that $p_2$ is an isomorphism since $F$ reflects isomorphisms.
\end{proof}
\begin{corollary} \label{cor:traingular-lemma}
	If $\C$ is a finitely complete category which satisfies the triangular lemma, then so does $\C \downarrow X$ and $X \downarrow \C$  for any object $X$. In particular, it follows that $\Pt_{\C}(X)$ satisfies the triangular lemma if $\C$ does.
\end{corollary}
\begin{proof}
	The proof follows from the fact that the codomain-assigning functor $X \downarrow \C \rightarrow \C$ and the domain-assigning functors $\C \downarrow X \rightarrow \C$ and $\Pt_{\C}(X) \rightarrow \C$ satisfy the conditions of Proposition~\ref{prop:points-triangular-lemma-persevation-reflection}. 
\end{proof}
\begin{remark}
	The trapezoid lemma \cite{Chajda2003}, which for varieties is equivalent to congruence distributivity, has also recently been studied in the categorical setting (see \cite{GranRodeloNgeufue2019}). The same argument above, also applies to the trapezoid lemma, so that if a finitely complete category $\C$ satisfies the trapezoid lemma, then so does $\C \downarrow X$, $X\downarrow \C$ and $\Pt_{\C}(X)$.
\end{remark}

\begin{corollary}
	Every finitely complete category $\C$ satisfying the triangular lemma is locally anticommutative. 
\end{corollary}
\begin{proof}
	By Corollary~\ref{cor:traingular-lemma} every category of points $\Pt_{\C}(X)$ satisfies the triangular lemma and is finitely complete, and the result follows by Remark~\ref{rem:pointed-traingular-lemma-anticommutative}.
\end{proof}
\begin{example}
Every finitely complete majority category is locally anticommutative. 
\end{example} 
\begin{proof}
If $\C$ is any finitely complete majority category, then $\Pt_{\C}(X)$ is a pointed finitely complete majority category (see Proposition~2.4 in \cite{Hoe18a}) for any object $X$ in $\C$, so that the result follows by Proposition~\ref{prop:pointed-majority-anticommutative}. 
\end{proof}
Much of what follows in the discussion below is just a slight adaption of the proof of Theorem 2.11 in \cite{Bou02} for local anticommutativity, and has been motivated from the results of \cite{Pedicchio1996}. Our main aim in what follows is to provide sufficient context for Remark~\ref{rem:modular-groupoid-distributive}, and for this reason we keep the length of the discussion minimal so as to be self-contained. 

Let $\C$ be a finitely complete category. For any object $X$ in $\C$, we may associate the object $(X \times X, \pi_1, \Delta_X)$ in $\Pt_{\C}(X)$, where $\pi_1$ is the canonical projection, and $\Delta_X$ is the diagonal morphism. Moreover, any equivalence relation $E$ represented by a monomorphism $e = (e_1, e_2)\colon  E_0 \rightarrow X \times X$ in $\C$ determines an object $(E_0, e_1, d_E)$ in $\Pt_{\C}(X)$ where $d_E\colon X \rightarrow E$ is the unique morphism with $e \circ d_E = \Delta_X$. Given any two equivalence relations $R$ and $S$ represented by the monomorphisms $r = (r_1, r_2)\colon R_0 \rightarrow X \times X$ and $s = (s_1, s_2)\colon S_0 \rightarrow X \times X$ respectively, then the morphisms $r$ and $s$ can be viewed as monomorphisms $(R_0, r_1, d_R) \xrightarrow{r} (X \times X, \pi_1, \Delta_X)$ and  $(S_0, s_1, d_S) \xrightarrow{s} (X \times X, \pi_1, \Delta_X)$ in $\Pt_{\C}(X)$. In what follows, we will say that the equivalence relations $R$ and $S$ \emph{commute}, if the monomorphisms $r$ and $s$ commute in $\Pt_{\C}(X)$. An equivalence relation is then said to be abelian if it commutes with itself.  Consider the pullback below:
\[
\xymatrix{
R_0 \times_X S_0 \ar[r]^-{p_1} \ar[d]_-{p_2} &  S_0 \ar[d]^{s_1} \\
R_0\ar[r]_{r_2} & X
}
\]
Set theoretically, the object $R_0 \times_X S_0$ may be considered as the subobject of $X^3$ consisting of all triples $(x,y,z)$ such that $x R y$ and $y S z$. Any morphism $p\colon  R_0 \times_X S_0 \rightarrow X$ satisfying $p(x,x,y) = y = p(y,x,x)$, i.e., any partial Mal'tsev operation $p\colon  R_0 \times_X S_0 \rightarrow X$, determines a cooperator for $r$ and $s$ in $\Pt_{\C}(X)$:  let $P$ be the product of $(R_0, r_1, d_R)$ and $(S_0, s_1, d_S)$ in $\Pt_{\C}(X)$, then the object of $P$, which is given by the pullback of $r_1$ along $s_1$ in $\C$, may too be presented set-theoretically as the subobject of $X^2 \times X^2$  consisting of all those pairs $((x,y), (x,z))$ such that $x R y$ and $x S z$. Now, given a morphism $p$ as above, the morphism $\phi\colon P \rightarrow X \times X$ defined by
\[
\phi ((x,y), (x,z)) = (x, p(y,x,z)),
\]
is readily seen to be a cooperator for $r$ and $s$ in $\Pt_{\C}(X)$. Conversely, any cooperator for $r$ and $s$ determines a partial  Mal'tsev operation $p\colon  R_0 \times_X S_0 \rightarrow X$ by the formula above.
\begin{proposition} \label{prop:local-anticommutative-commuting-equivalence-relations}
	Let $\C$ be a finitely complete locally anticommutative category. If $R$ and $S$ are any equivalence relations on the same object $X$ in $\C$ which commute, then $R \cap S = \Delta_X$.
\end{proposition}
\begin{proof}
This is immediate since $R \cap S$ is represented by the pullback of $r$ along $s$ in $\Pt_{\C}(X)$ (which must be the zero object since $\Pt_{\C}(X)$ is anticommutative), and the zero object in $\Pt_{\C}(X)$ is $(X,1_X, 1_X)$.
\end{proof} 
\begin{remark} \label{rem:connecting-equivalence-relations}

If the equivalence relations $R$ and $S$ admit a \emph{connector} in the sense of \cite{BouGra02}, i.e., $R$ and $S$ are \emph{connected}, then $R$ and $S$ commute. This is because a connector between $R$ and $S$ is a certain partial Mal'tsev operation $p:R_0 \times_X S_0 \rightarrow X$ satisfying some additional requirements.  
\end{remark}
\begin{definition} \label{def-internal-category}
An internal category $C$ in a category $\C$ is a diagram
\[
\xymatrix{
C_2 \ar@<2ex>[r]^{p_1}  \ar@<-2ex>[r]_{p_2} \ar[r]|m  & C_1 \ar@<2ex>[r]^{d_1}  \ar@<-2ex>[r]_{d_2} & C_0 \ar[l]|{s}
}
\]
in $\C$ where the square
\[
\xymatrix{
C_2\ar[r]^{p_2} \ar[d]_{p_1} & C_1 \ar[d]^{d_1}\\
C_1 \ar[r]_{d_2} & C_0
}
\]
is a pullback, and we have the following relations
\begin{enumerate}[(i)]
\item $d_1 \circ s = 1_{C_0} =  d_2 \circ s$
\item $m \circ (1_{C_1}, s\circ d_2) = 1_{C_1} = m \circ (s\circ d_1,  1_{C_1})$
\item $d_1 \circ p_1 = d_1 \circ m$ and $d_2 \circ p_2 = d_2 \circ m$
\item $m \circ (p_1\circ q_1, m \circ q_2) = m \circ (m \circ q_1, p_2\circ q_2)$ where 
\[
\xymatrix{
C_3 \ar[d]_{q_1} \ar[r]^{q_2}& C_2 \ar[d]^{p_1} \\
C_2 \ar[r]_{p_2} & C_1
}
\]
is a pullback, and $(p_1\circ q_1, m \circ q_2), (m \circ q_1, p_2\circ q_2):C_3 \rightarrow C_2$ are the morphisms induced by the pullback $(C_2, p_1,p_2)$. 
\end{enumerate}
\end{definition} 
\begin{definition} \label{def-internal-groupoid}
An internal category $C$ as in Definition~\ref{def-internal-category} is called an \emph{internal groupoid} if there exists an ``inverse'' morphism $\sigma: C_1 \rightarrow C_1$ satisfying the following:
\[
d_1 \circ \sigma = d_2, \quad d_2 \circ \sigma =  d_1 ,
\]
and 
\[
m \circ (1_{C_1}, \sigma) = s \circ d_1, \quad  m \circ (\sigma, 1_{C_1}) = s \circ d_2.
\]
\end{definition}
\begin{corollary} \label{cor:internal-groupoid}
Given any internal groupoid $C$ as in Definition \ref{def-internal-category} and Definition~\ref{def-internal-groupoid} in any finitely complete locally anticommutative category $\C$, the morphism $(d_1, d_2):C_1 \rightarrow C_0 \times C_0$ is a monomorphism, i.e., any internal groupoid in a finitely complete locally anticommutative category is an equivalence relation.
\end{corollary}
\begin{proof}
The groupoid structure determines a canonical connector between the equivalence relations $\Eq(d_1)$ and $\Eq(d_2)$ (see Example 1.4 in \cite{BouGra02}), so that in particular we have that $\Eq(d_1)$ and $\Eq(d_2)$ are connected, and hence $\Eq(d_1) \cap \Eq(d_2) = \Delta_{C_1}$ by Proposition~\ref{prop:local-anticommutative-commuting-equivalence-relations} and Remark~\ref{rem:connecting-equivalence-relations}. Then we have that $\Eq(d_1, d_2) = \Eq(d_1) \cap \Eq(d_2) = \Delta_{C_1}$ and hence $(d_1, d_2)$ is a monomorphism.
\end{proof}
\begin{remark} \label{rem:modular-groupoid-distributive}
In the paper \cite{Pedicchio1998}, M.C. Pedicchio showed that a congruence modular variety $\V$ is congruence distributive if and only if every internal groupoid in $\V$ is an equivalence relation (see Theorem~3.2 in \cite{Pedicchio1998}). Note, this fact was already announced in \cite{JanPed97}. In particular, it follows from Corollary~\ref{cor:internal-groupoid} that a congruence modular variety is congruence distributive if and only if it is locally anticommutative. By Theorem~\ref{thm:local-Mal'tsev-characterization} (iii) below, we have that a variety which is congruence modular and satisfies the triangular lemma on pullbacks is congruence distributive. In particular, this generalizes the equivalence of (a) and (d) of Theorem~1 in the paper \cite{Chajda2003}.
\end{remark}
\section{Characterization of anticommutative and locally anticommutative varieties}
Let $\V$ be a variety of universal algebras, and suppose that $f\colon A \rightarrow X$ and $g\colon B \rightarrow X$ are two homorphisms in $\V$. The pullback $A \times_X B$ of $f$ along $g$ is given by
\[
A \times_{X} B = \{(x,y) \mid f(x) = g(y)\},
\] 
together with the projections $p_1:(x,y) \mapsto x$ and $p_2:(x,y) \mapsto y$. As mentioned earlier, we note that the kernel congruence relation $\Eq(f)$ of a homomorphism $f$ in $\V$ is simply given by the pullback of $f$ along itself. In the proofs of Theorem~\ref{thm:Mal'tsev-characterization} and Theorem~\ref{thm:local-Mal'tsev-characterization}, we make use 2$\times$2 matrices to represent elements of a congruence $\Theta$ on the pullback $A \times_X B$. In particular, this is represented as follows:
\[
\mat{a}{b}{a'}{b'} \in \Theta \Longleftrightarrow (a,b) \Theta (a',b').
\]
\begin{theorem} \label{thm:Mal'tsev-characterization}
	For a pointed variety of universal algebras $\V$, the following are equivalent. 
	\begin{enumerate}
		\item $\V$ is anticommutative. 
		\item $\V$ admits unary terms $u_1,\dots,u_m$ and $v_1,\dots,v_m$, as well as $(m+2)$-ary terms $p_1,\dots,p_n$ satisfying the equations:
		\begin{itemize}
			\item $p_1(u_1(x),\dots,u_m(x),x,0) = x$ and $p_1(v_1(x),\dots,v_m(x),0,x) = 0$.
			\item $p_{i+1}(u_1(x),\dots,u_m(x),x,0) = p_{i}(u_1(x),\dots,u_m(x),0,x)$. 
			\item $p_{i+1}(v_1(x),\dots,v_m(x),0,x) = p_{i}(v_1(x),\dots,v_m(x),x,0)$. 
			\item $p_n(u_1(x),\dots,u_m(x),0,x) = 0$ and $p_n(v_1(x),\dots,v_m(x),x,0) = 0$.
		\end{itemize}
	\end{enumerate}
\end{theorem}
\begin{proof}
	Let $\V$ be a pointed variety which satisfies $(ii)$ of Proposition~\ref{prop:characterization-one}, and let $\Theta$ be the congruence on  $F_{\V}(x) \times F_{\V}(x)$, generated by $(x,0) \Theta (0,x)$. Then in the quotient of $F_{\V}(x) \times F_{\V}(x)$ by $\Theta$, the element $(x,0)$ will be identified with $(0,0)$, so that $(x,0) \Theta (0,0)$. But $\Theta$ may be obtained by first closing the relation 
	\[
	\{\mat{x}{0}{0}{x}, \mat{0}{x}{x}{0}\},
	\]
	under reflexivity, then under all operations, and then under transitivity. Therefore, there exists elements $z_0,\dots,z_n$ such that $z_0 = (x,0)$ and $z_n = (0,0)$ and $z_{i-1} \Theta z_{i}$ for all $0<i$. For each $0< i$ we have that $(z_{i-1}, z_{i})$ has the form:
	\[
	 p_i\big( \mat{u_{1,i}(x)}{v_{1,i}(x)}{u_{1,i}(x)}{v_{1,i}(x)}, \dots ,\mat{u_{m_i, i}(x)}{v_{m_i, i} (x)}{u_{m_i, i}(x)}{v_{m_i, i}(x)},\mat{x}{0}{0}{x}, \mat{0}{x}{x}{0}\big),
	\]
	where the $u$'s and the $v$'s are unary terms. We may assume without loss of generality that $m_i = m$, and moreover, we may also assume that $u_{k, i} = u_k$ as well as $v_{k,i} = v_{k}$ for any $i = 1,\dots,m$. Then writing out the equalities component-wise gives the equations in the statement of the theorem. 
\end{proof}

\begin{remark} \label{rem:Mal'tsev-conditions-antilinear}
	Recall that a pointed variety of universal algebras $\V$ is unital if and only if it admits a J\`onsson-Tarski operation (i.e., a binary term $+$ satisfying $x + 0 = x = 0 + x$). Therefore, by Proposition~\ref{prop:regular-characterization} a variety of universal algebras is antilinear if and only if it admits a J\`onsson-Tarski operation and the terms of Theorem~\ref{thm:Mal'tsev-characterization}. 
\end{remark}
\noindent
\subsection{Characterization of locally anticommutative varieties}
Let $\C$ be any category with pullbacks, and let $(A,f,s)$ be any object of $\Pt_{\C}(X)$. Consider the diagram below where the square is a pullback.
\[
\xymatrix{
& X \ar[rd]|{(s,s)}\ar@/^1pc/@{->}[rrd]^{s} \ar@/_1pc/@{->}[rdd]_{s} & & \\
& & \Eq(f) \ar[r]^{p_2} \ar[d]_{p_1} \ar[dr]|d & A \ar[d]^{f} \\
& & A \ar[r]_{f} & X
}
\]
A binary product of $(A,f,s)$ with itself in $\Pt_{\C}(X)$ is given by $(\Eq(f), d, (s,s))$ together with the morphisms $p_1$ and $p_2$. If $\C$ has coequalizers, then so does $\Pt_{\C}(X)$, and moreover, they are computed as they are in $\C$. Using these stated facts, it is then possible to state a local version of Proposition~\ref{prop:characterization-with-coequalizers}, which is the content of the following proposition. 
\begin{proposition} \label{prop:locally-anticommutative-coequalizer-characterization}
Let $\C$ be a category with pullbacks and coequalizers. Then $\C$ is locally anticommutative if and only if for any split epimorphism $f: A \rightarrow X$ with splitting $s:X \rightarrow A$ and any coequalizer diagram
	\[
	\xymatrix{
		A \ar@<-.5ex>[r]_-{(1, s \circ f)} \ar@<.5ex>[r]^-{(s \circ f, 1)}  & \Eq(f) \ar@{->>}[r]^{q} & Q,
	}
	\]
	we have $q(s \circ f, 1) = q(s\circ f, s \circ f) = q(1, s \circ f)$.
\end{proposition}
\noindent
Recall from the introduction that we say a variety $\V$ satisfies the triangular lemma on pullbacks, if for any two morphisms $f\colon A \rightarrow X$ and $g\colon B \rightarrow X$, and any congruence $\Theta$ on the pullback $A \times_X B$ of $f$ along $g$ we have: 
\[
(x,y)\Theta (x',y') \implies (x',y) \Theta (x',y'),
\]
for any $(x,y), (x',y), (x',y') \in A \times_X B$. 
		\[
		\xymatrix{
			& (x',y') \ar@{-}[d]|{\Eq(p_1)} \ar@{..}@/^2pc/[d]^{\Theta} \\
			(x,y)    \ar@{-}[ur]^{\Theta} \ar@{-}[r]_{\Eq(p_2)}  & (x',y) 
		}
		\]
\begin{theorem} \label{thm:local-Mal'tsev-characterization}
	For a variety $\V$ of universal algebras, the following are equivalent: 
	\begin{enumerate}[(i)]
		\item $\V$ is locally anticommutative.
		\item $\V$ admits binary terms $b_1,\dots,b_m$ and $c_1,\dots,c_m$ as well as $(m+2)$-ary terms $p_1,\dots,p_n$ such that:
		\begin{itemize}
			\item $p_1(b_1(x,y),\dots,b_m(x,y),x,y) = x$, 
			\item $p_{1}(c_1(x,y),\dots,c_m(x,y),y,x) = y$, 
			\item $p_{i+1}(b_1(x,y),\dots,b_m(x,y),x,y) = p_{i}(b_1(x,y),\dots,b_m(x,y),y,x)$,
			\item $p_{i+1}(c_1(x,y),\dots,c_m(x,y),y,x) =  p_{i}(c_1(x,y),\dots,c_m(x,y),x,y)$,
			\item $p_n(b_1(x,y),\dots,b_m(x,y),y,x) = x$, 
			\item $p_n(c_1(x,y),\dots,c_m(x,y),x,y) = x$.
			\item $b_i(x,x) = c_i(x,x)$
		\end{itemize}
		\item $\V$ satisfies the triangular lemma on pullbacks.
	\end{enumerate}
\end{theorem}

\begin{proof}
	For $(i) \implies (ii)$, consider the free algebra $F_{\V}(x,y)$ in $\V$ over the set $\{x,y\}$. Let $f\colon F_{\V}(x,y) \rightarrow F_{\V}(z)$ and $s\colon  F_{\V}(z) \rightarrow F_{\V}(x,y)$  be the unique morphisms with $f(x) = z = f(y)$ and $s(z) = x$, respectively. Consider the coequalizer diagram
	\[
	\xymatrix{
		F_{\V}(x,y) \ar@<-.5ex>[r]_{(1, s \circ f)} \ar@<.5ex>[r]^{(s \circ f, 1)}  & \Eq(f) \ar@{->>}[r]^{q} & Q.
	}
	\]
	Applying Proposition~\ref{prop:locally-anticommutative-coequalizer-characterization} to the diagram above, it follows that $(x,y)\Eq(q) (x,x)$. The coequalizer $q$ is obtained by the quotient of $\Eq(f)$ by the congruence generated by the pairs $(s \circ f, 1)(b(x,y)), (1, s \circ f)(b(x,y))$ where $b$ is any binary term in the theory of $\V$. It is straightforward to verify that
	\[
	\big((s \circ f, 1)(b(x,y)), (1, s \circ f)(b(x,y))\big) = b\big(((x,x), (x,x)), ((x,y), (y,x))\big),
	\]
	and hence that $\Eq(q)$ is the principle congruence generated by $(x,y)\Eq(q) (y,x)$. Therefore, there exists elements $z_0,\dots,z_n$ such that $z_0 = (x,y)$ and $z_n = (x,x)$ and $z_{i-1} \Theta z_{i}$ for $0 < i$. Moreover, we have that $(z_{i-1}, z_{i})$ is equal to
	\[
	\small 
	 p_i\big( \mat{b_{1}(x,y)}{c_{1}(x,y)}{b_{1}(x,y)}{c_{1}(x,y)}, \dots ,\mat{b_{m}(x,y)}{c_{m} (x,y)}{b_{m}(x,y)}{c_{m}(x,y)},\mat{x}{y}{y}{x}, \mat{y}{x}{x}{y}\big),
	\]
	for all $0 < i$ and where $(b_i(x,y), c_i(x,y)) \in \Eq(f)$. Then $b_i(z,z) = c_i(z,z)$ and writing out the equations determined by $z_i \Theta z_{i+1}$ component-wise we get the equations in the statement of the theorem. For $(ii) \implies (iii)$, suppose that $A \times_X B$ is the pullback of a morphism $f:A \rightarrow X$ along another morphisms $g:B \rightarrow X$ in $\V$, and suppose that $(x,y), (x,y'), (x', y') \in A \times_X B$. Then $f(x') = f(x) = g(y) = g(y')$, and hence $f(b_i(x, x')) = b_i(f(x),f(x')) = c_i(g(y),g(y') = g(c_i(y,y'))$ so that $(b_i(x,x'), c_i (y,y')) \in A\times_X B$ for any $i = 1,\dots m$. Define the elements $(z_{0,i}, z_{1,i})$ of $\Theta$ so that $(z_{0,i}, z_{1,i})$ is equal to
	\[
	 p_i\big( \mat{b_{1}(x,x')}{c_{1}(y,y')}{b_{1}(x,x')}{c_{1}(y,y')}, \dots ,\mat{b_{m}(x,x')}{c_{m} (y,y')}{b_{m}(x,x')}{c_{m}(y,y')},\mat{x}{y}{x'}{y'}, \mat{x'}{y'}{x}{y}\big).
	\]
	Then the equations in the statement of the theorem ensure that $z_{0,1} = (x,y)$ and $z_{1,n} = (x', y)$, and that $z_{1,i} = z_{0, i+1}$ so that by the transitivity of $\Theta$ we have $(x,y) \Theta (x',y)$. For $(iii) \implies (i)$, we have to show that for any split epimorphism $p\colon  A \rightarrow X$ with splitting $s\colon  X \rightarrow A$, and any coequalizer 
	\[
	\xymatrix{
		X\ar@<-.5ex>[r]_-{(1, s \circ p)} \ar@<.5ex>[r]^-{(s \circ p, 1)} & \Eq(p) \ar@{->>}[r]^q & Q ,
	}
	\]
	that $q(s \circ p, 1) = q(s \circ p, s \circ p)$ according to Proposition~\ref{prop:locally-anticommutative-coequalizer-characterization}. But for any $x \in X$ we have that $(s\circ p (x), x), (x, s\circ p (x)), (s\circ p (x), s \circ p(x)) \in \Eq(p)$ and moreover we have $(x, s\circ p (x)) \Eq(q) (s\circ p (x), x)$ which implies by $(iii)$ that $(s \circ p(x), x)\Eq(q)(s \circ p(x), s \circ p(x))$ for any $x$, and the result follows. 
\end{proof}
\section{Concluding remarks}
Recall from earlier that a variety $\V$ is said to have \emph{directly decomposable congruence classes} (DDCC) \cite{Duda1986} if every congruence class $C$ on a product $X \times Y$ in $\V$ is such that $C = \pi_1(C) \times \pi_2(C)$. It was shown in \cite{Duda1986} that (DDCC) is also a Mal'tsev property, which turns out to be equivalent to the existence of binary terms $b_1,\dots,b_m$ and $c_1,\dots,c_m$ as well as $(m+2)$-ary terms $p_1,\dots,p_n$ such that:
		\begin{itemize}
			\item $p_1(b_1(x,y),\dots,b_m(x,y),x,y) = x$,
			\item $p_{1}(c_1(x,y),\dots,c_m(x,y),y,x) = y$, 
			\item $p_{i+1}(b_1(x,y),\dots,b_m(x,y),x,y) = p_{i}(b_1(x,y),\dots,b_m(x,y),y,x)$,
			\item $p_{i+1}(c_1(x,y),\dots,c_m(x,y),y,x) =  p_{i}(c_1(x,y),\dots,c_m(x,y),x,y)$,
			\item $p_n(b_1(x,y),\dots,b_m(x,y),y,x) = x$,
			\item $p_n(c_1(x,y),\dots,c_m(x,y),x,y) = x$.
		\end{itemize}
The above terms for (DDCC) differ from the terms in  Theorem~\ref{thm:local-Mal'tsev-characterization} by only one extra requirement, namely, that the $b_i(z,z) = c_i(z,z)$. This observation immediately gives us the following two corollaries of Theorem~\ref{thm:local-Mal'tsev-characterization}.
\begin{corollary}
Every locally anticommutative variety has \emph{(DDCC)}, and an idempotent variety is locally anticommutative if and only if it has \emph{(DDCC)}.
\end{corollary}
\begin{remark}
The above corollary implies that every locally anticommutative variety of algebras has \emph{difunctional class relations} in the sense of \cite{HoefnagelJanelidzeRodelo2020}.
\end{remark}
\noindent
Together with Remark~\ref{rem:modular-groupoid-distributive} we also have:
\begin{corollary}
An idempotent congruence modular variety is congruence distributive if and only if it has directly decomposable congruence classes. 
\end{corollary}


\begin{thebibliography}{10}

\bibitem{CarPedPir92} Carboni, A., Pedicchio, M.C., Pirovano, N.: Internal graphs and internal groupoids in Mal'tsev
categories. Canadian Math. Soc. Conference proceedings 13, 97--109 (1992)

\bibitem{CarLamPed91} Carboni, A., Lambek, J., Pedicchio, M.C.: Diagram chasing in {Mal'cev} categories. J. Pure Appl. Algebra \textbf{69}, 271--284 (1991).

\bibitem{BGO71}
Barr, M., Grillet, P.A., van Osdol, D.H.: Exact categories and categories of sheaves.
 Springer, Lecture Notes in Mathematics \textbf{236} (1971).

\bibitem{Bor94b}
F.~Borceux.: Handbook of Categorical Algebra 2. Cambridge University Press (1994).

\bibitem{BorceuxBourn}
Borceux, F., Bourn, D.: Mal’cev, protomodular, homological and semi-abelian categories. Mathematics and Its Applications (2004).

\bibitem{Bou96}
Bourn., D.: Mal'cev categories and fibration of pointed objects. Appl. Categ. Structures \textbf{4}, 307--327 (1996)

\bibitem{Bou01}
Bourn., D.: A categorical genealogy for the congruence distributive property. Theory Appl. Categ. \textbf{8}, 391--407 (2001)

\bibitem{Bou02}
Bourn., D.: Intrinsic centrality and associated classifying properties. J. Algebra \textbf{256}, 126--145 (2002)

\bibitem{BouGra02}
Bourn, D., Gran, M.: Centrality and connectors in {Maltsev} categories. Algebra Universalis \textbf{48}, 309--331 (2002)

\bibitem{CLW93}
Carboni, A., Lack, S., Walters, R.F.C.: Introduction to extensive and distributive categories. J. Pure Appl. Algebra \textbf{84}, 145--158 (1993)

\bibitem{Chajda2003} Chajda, I., Cz\'edli, G.,  Horva\'ath, E.K.: Trapezoid lemma and congruence distributivity. Math. Slovaca \textbf{53}, 247--253 (2003)

\bibitem{Duda1986} Duda, J.: Varieties having directly decomposable congruence classes. \v Casopis pro p\v estov\'an\'i matematiky \textbf{111}, 394--403 (1986)

\bibitem{Duda2000} Duda, J.: The triangular principle for congruence distributive varieties. Abstract of a lecture seminar presented in Brno, March (2000)

\bibitem{FraserHorn1970} Fraser, G.A., Horn, A.: Congruence relations in direct products. Proc. Amer. Math. Soc. \textbf{26}, 390--394 (1970)

\bibitem{Gum83} Gumm, H.P.: Geometrical methods in congruence modular algebras.
Mem. Amer. Math. Soc. \textbf{45} (1983)

\bibitem{Chjada2007} Chajda, I., Halas, R.: Varieties satisfying the triangular scheme need not be conguence distributive. Acta Univ. Palack. Olomuc. Fac. Rerum Natur. Math. \textbf{46}, 19--24 (2007)

\bibitem{Hoe18a}
Hoefnagel, M.: Majority categories. Theory Appl. Categ. \textbf{34}, 249--268 (2019)

\bibitem{Hoefnagel2019c}
Hoefnagel, M.: Products and coequalizers in pointed categories. Theory Appl. Categ. \textbf{34}, 1386--1400 (2019)

\bibitem{Hoe2020a}
Hoefnagel, M.: $\M$-coextensive objects and the strict refinement property. J. Pure Appl. Algebra \textbf{224}, 106381 (2020) 


\bibitem{Hoe18b} Hoefnagel, M.: Characterizations of majority categories. Appl. Categ. Structures \textbf{28}, 113--134 (2020)

\bibitem{HoefnagelJanelidzeRodelo2020} Hoefnagel, M., Janelidze, Z., Rodelo, D.: On difunctionality of class relations. Algebra Universalis \textbf{81}, 19 (2020). https://doi.org/10.1007/s00012-020-00651-z

\bibitem{Huq1968}
Huq, S.A.: Commutator, nilpotency, and solvability in categories.  Q. J. Math \textbf{19}, 363--389 (1968)

\bibitem{Chjadaetal2003} Chajda, I., Eigenthaler, G., L\"anger, H.: Congruence Classes in Universal Algebra. Research Expositions in Mathematics \textbf{26}, Heldermann Verlag (2003)

\bibitem{JanPed97}
Janelidze, G., Pedicchio, M.C.: Internal categories and groupoids in congruence modular varieties. J. Algebra \textbf{193}, 552--570 (1997)

\bibitem{Janelidze2003} Janelidze, Z.: Characterization of pointed varieties of universal algebras with
  normal projections.  Theory Appl. Categ. \textbf{11}, 212--214 (2003)

\bibitem{Janelidze2004}
Janelidze, Z.: Varieties of universal algebras with normal local projections. Georgian Math. J. \textbf{11}, 93–--98 (2004)

\bibitem{Janelidze2005}
Janelidze, Z.: Subtractive categories. Appl. Categ. Structures \textbf{13}, 343--350 (2005)

\bibitem{GranRodeloNgeufue2019}
 Gran, M., Rodelo, D., Tchoffo Nguefeu, I.: Facets of congruence distributivity in Goursat categories. J. Pure Appl. Algebra \textbf{224}, 106380 (2020) 

\bibitem{Nem65} Nemitz, W.C.: Implicative semi-lattices.  Trans. Amer. Math. Soc. \textbf{117}, 128--142 (1965)

\bibitem{Pedicchio1996} Pedicchio, M.C.: Arithmetical categories and commutator theory.  Appl. Categ. Structures \textbf{4}, 307--327 (1996)
 
\bibitem{Pedicchio1998}
Pedicchio, M.C.: Some remarks on internal pregroupoids in varieties.
Comm. Algebra \textbf{26}, 1737--1744 (1998)
\end{thebibliography}
\end{document}